\documentclass[12pt,reqno]{amsart}
\usepackage{graphicx}
\usepackage{amssymb,amsmath}
\usepackage{amsthm}
\usepackage[pdf]{pstricks}
\usepackage{color,graphicx}
\usepackage{eucal}
\usepackage{hyperref}
\usepackage{color}
\usepackage{epstopdf}

\newtheorem{teo}{Theorem}[section]
\newtheorem{lema}{Lemma}[section]
\newtheorem{prop}{Proposition}[section]
\newtheorem{defi}{Definition}[section]
\newtheorem{obs}{Remark}[section]
\newtheorem{coro}{Corollary}[section]

\newcommand{\be}{\begin{equation}}
\newcommand{\ee}{\end{equation}}
\newcommand{\R}{{\mathbb R}}

\newcommand{\di} {\displaystyle}

\begin{document}
\vglue-1cm \hskip1cm
\title[Periodic Waves for the Cubic-Quintic NLS]{Periodic Waves for the Cubic-Quintic NonLinear Schrodinger Equation: Existence and Orbital Stability}

\begin{center}
	
	

	\maketitle
	
	{\bf Giovana Alves }
	
	{Centro de Ci\^encias Exatas, Naturais e Tecnol\'ogicas\\
		Universidade Estadual da Regi\~ao Tocantina do Maranh\~ao\\
		R. Godofredo Viana, 1300, CEP 65900-000, Imperatriz, MA, Brazil.}\\
	{ giovana.alves@uemasul.edu.br}
	
	\vspace{3mm}

	{\bf F\'abio Natali}
	
	{Departamento de Matem\'atica\\
		Universidade Estadual de Maring\'a\\
		Avenida Colombo, 5790, CEP 87020-900, Maring\'a, PR, Brazil.}\\
	{fmanatali@uem.br}

\end{center}	
	
	\begin{abstract}
		In this paper, we prove existence and orbital stability results of periodic standing waves for the cubic-quintic nonlinear Schr\"odinger equation. We use the implicit function theorem to construct a smooth curve of explicit periodic waves with \textit{dnoidal} profile and such construction can be used to prove that the associated period map is strictly increasing in terms of the energy levels. The monotonicity is also useful to obtain the behaviour of the non-positive spectrum for the associated linearized operator around the wave. Concerning the stability, we prove that the dnoidal waves are orbitally stable in the energy space restricted to the even functions.
	
	\end{abstract}
	
	\section{Introduction}
	\label{intro}
	In this work, we establish the existence of explicit solutions and orbital stability results
	of positive and periodic standing waves for the nonlinear Schr\"{o}dinger equation
	\begin{equation}
		iu_t+u_{xx}+|u|^2u+|u|^4u=0, \label{eq21}
	\end{equation}
	where $u:\mathbb{R} \times \mathbb{R}_+ \rightarrow \mathbb{C}$ is an $L$-periodic function in the spatial variable. Equation $(\ref{eq21})$ is a particular case of a more general equation
	\begin{equation}
		iu_t+u_{xx}+a|u|^2u+b|u|^4u=0, \label{eq2}
	\end{equation}
	where $a,b$ are real constants satisfying $a^2+b^2\neq0$. %(but in our stability analysis, we shall consider $a=b=1$).
	In a general context, the nonlinear Schr\"odinger equation has many applications in physics and engineering as in nonlinear optics, quantum mechanics, and nonlinear waves. More specifically, the cubic-quintic NLS equation (CQNLS henceforth) $(\ref{eq2})$ appears in the interaction of boson gas and nonlinear optics (see \cite{BGMP}). In a mathematical point of view, the cubic-quintic nonlinearity in (\ref{eq2}) is an interesting model since the double power gives us the absence of scaling invariance. This fact establishes some difficulties for obtaining the critical exponent in $L^2$, a very important feature concerning the existence of global solutions in the energy space $H^1$. It is well know that either $a=1$ and $b=0$ or $a=0$ and $b=1$, the NLS equation $(\ref{eq2})$ enjoys this important property and, in the latter case, it is called \textit{mass-critical} NLS equation.

	Standing wave solutions for the equation (\ref{eq2}) are given by the formula
	\begin{equation}\label{PW2}
		u(x,t)=e^{i\omega t}\phi(x),  
	\end{equation}
	where $\omega\in \R$ and $\phi$ is a  $L$-periodic smooth function. By substituting this kind of solution into (\ref{eq2}), we obtain the second order nonlinear ordinary differential equation
	\begin{equation}
		-\phi''+\omega\phi -a\phi^3-b\phi^5=0. \label{edo}
	\end{equation}
	
	Formally, equation (\ref{eq2}) conserves the energy
	\begin{equation}
		E(u)=\frac{1}{2}\int_{0}^{L}|u_x|^2-\frac{a|u|^4}{2}-\frac{b|u|^6}{3}dx \label{E}
	\end{equation}
	and the mass
	\begin{equation}
		F(u)=\frac{1}{2}\int_{0}^{L}|u|^2dx. \label{F}
	\end{equation}
	
	In addition, equation (\ref{eq2}) can be seen as a (real) Hamiltonian system which is a good feature to study the orbital stability of standing waves. In fact, writing $u= P + iQ$ and separating real and imaginary parts, we see that (\ref{eq2}) can be reduced to a single evolution system of equations as
	\begin{equation}
		\frac{d}{dt}U(t)=JE'(U) \label{ham}
	\end{equation}
	where $E'$ represents the Fr\'echet derivative of $E$ with respect to $U=\left( \begin{array}{c}
		P \\
		Q \\
	\end{array}\right)=(P,Q)$, and $J$ is the skew-symmetric matrix
	$J= \left( \begin{array}{cc}
		0& 1 \\
		-1 & 0  \\
	\end{array}\right).$ Using (\ref{edo}), we see that $E'(\phi,0)+\omega F'(\phi,0)=0$, that is, $\Phi=(\phi, 0)$ is a critical point of the Lyapunov functional $G=E + \omega F$. As far as we know, the orbital stability of the wave $\Phi$ can be determined by minimizing the functional $E$ under a fixed constrained momentum $F$. Thus, since $\Phi$ is a critical point of the smooth functional $G$, it is intuitive to think that the "stability" can be determined by proving that the second derivative of $G$ at the point $\Phi$, and denoted by $G''(\Phi)$, is strictly positive in some sense. To be more specific, let us consider the linear operator
	\begin{equation}\mathcal{L}= \left( \begin{array}{cc}
			-\partial_x^2+\omega-3a\phi^2-5b\phi^4& 0 \\
			0 & -\partial_x^2+\omega-a\phi^2-b\phi^4  \\
		\end{array}\right).\label{L-35i} 
	\end{equation}
	\indent It is possible to show $\mathcal{L}=G''(\Phi)$ and we obtain by $(\ref{edo})$ that $\mathcal{L}(\phi',0)=(0,0)$ and $\mathcal{L}(0,\phi)=(0,0)$, that is, $\mathcal{L}$ has at least two null directions. Moreover, denoting $(\cdot,\cdot)_{\mathbb{L}_{per}^2}$ the inner product in $\mathbb{L}_{per}^2=L_{per}^2\times L_{per}^2$ (see more information for notations in the next section) and considering $a,b>0$, we obtain $(\mathcal{L}\Phi,\Phi)_{\mathbb{L}_{per}^{2}}=-2a\int_0^L\phi^4-4b\int_{0}^L\phi^6dx<0$ and $\mathcal{L}$ has negative directions at the wave $\Phi$. The question which naturally arises is: \textit{how to obtain the positivity of $\mathcal{L}$ by considering this inconvenient scenario with null and negative directions?} To overcome this difficult and at least in our context, the pioneer work in \cite{GSS} established that in a non-favourable setting as presented above, it is possible to obtain the positiveness of $\mathcal{L}$ by taking, for example, only one negative direction which needs to be "compensated" with only one positive direction of the hessian matrix associated to the function $d(\omega)=E(\Phi)+\omega F(\Phi)$. In addition, the quantity of null directions need to be considered according to the quantity of symmetries present in the evolution equation. In our context and since we are considering standing waves of the form $u(x,t)=e^{i\omega t}\phi(x)$, we need to consider only the rotation symmetry for the equation $(\ref{eq2})$, so that the quantity of null directions needs to be one (see Remark $\ref{obs12}$ ahead). However, besides the rotation symmetry, it is well known that $(\ref{eq2}) $ is invariant under translations acting in the whole energy space $\mathbb{H}_{per}^1$. Thus, it is convenient to consider a suitable space where the translation invariance fails. Considered here will be $\mathbb{H}_{per}^1$ restricted to the even periodic functions and denoted by $\mathbb{H}_{per,e}^1$.\\
	\indent Summarizing our considerations for obtaining the orbital stability of standing waves in our context, we need to obtain the following \textit{sufficient set of conditions}: 
	\begin{itemize}
		\item[(i)] there exists a smooth curve of solution for \eqref{edo}, $\omega \in I\subset\mathbb{R}  \mapsto\phi_{\omega} \in H_{per,e}^2([0,L]),$ where each $\phi:=\phi_{\omega}$ has period $L$.
		\item[(ii)] The linearized operator  $\mathcal{L}$ defined in $(\ref{L-35i})$ and restricted to the space of even periodic functions has only one negative eigenvalue which is simple and zero is a simple eigenvalue associated to the eigenfunction $(0,\phi)$.
		\item[(iii)] The hessian matrix of $d: I \rightarrow \mathbb{R}$ and defined by $d(\omega)=E(\Phi)+\omega F(\Phi)$ is positive definite, that is, \begin{equation}\label{d2}
			d''(\omega)=\frac{1}{2}\frac{d}{d\omega}\int_0^{L}\phi^2dx>0.
		\end{equation}
		
	\end{itemize}
	In addition, according to the conditions above and since operator $J$ in $(\ref{ham})$ is invertible with bounded inverse, we can conclude from \cite{GSS} that the wave $\phi$ is orbitally unstable if $d''(\omega)<0$.\\
	\indent We describe how to obtain (i)-(iii) in our case. First, we consider equation $(\ref{eq21})$ and the reason for that is to study the orbital stability of positive and symmetric periodic waves which turns around the equilibrium point and they are bounded by the homoclinic solution. As we will see below, cases $a=0$, $b=1$ and $a=1$, $b=0$ have the same kind of periodic solutions and we describe with details the results concerning these cases in the next paragraphs.\\
	\indent Item (i) is obtained by using the quadrature method and the implicit function theorem (see \cite{JaimeFabio1} and \cite{cesar}). We construct, for a fixed period $L>0$, a smooth curve of periodic waves $\phi$ depending on $\omega$ with dnoidal profile as
	\begin{equation}
		\phi(x)= \frac{\sqrt{\alpha_3}dn\left(\frac{2}{\sqrt{3}g}x,k\right)}{\sqrt{1+\beta^2sn^2\left(\frac{2}{\sqrt{3}g} x,k\right)}} \label{sol21}
	\end{equation}
	where parameters $g$, $\beta$ and the modulus $k$ are 
	\be\label{kg1} g=\frac{2}{\sqrt{\alpha_3(\alpha_2-\alpha_1)}}, \quad \beta^2=-\frac{\alpha_3}{\alpha_1}k^2>0,\quad k^2=\frac{-\alpha_1(\alpha_3-\alpha_2)}{\alpha_3(\alpha_2-\alpha_1)}.\ee
	Parameters $\alpha_i$, $i=1,2,3$, are the non-zero roots of the polynomial $P(s)=-s^4-\frac{3s^3}{2}+3\omega s^2+3Bs$ where $B$ is the second constant of integration which appears in the quadrature form associated to the equation $(\ref{edo})$ given by
	\be
	(\phi')^2=-\frac{\phi^6}{3}-\frac{\phi^4}{2}+\omega\phi^2+B, \label{quadratura-351}
	\ee
	\indent Another important fact obtained by the construction of smooth periodic waves is the spectral information required in item (ii). In  fact, the construction of smooth periodic waves is crucial in our analysis since one obtains that the period map $\Psi:\Omega\rightarrow\R$ defined in a convenient open subset $\Omega\subset\mathbb{R}^2$ and given by
	\be\label{period-map1}
	\Psi(\alpha,\omega)=\frac{\sqrt{8} {{3}^{\frac{1}{4}}}  }{\sqrt{{\alpha}} {{\left( 16\omega-4\alpha^2-4\alpha+3\right) }^{\frac{1}{4}}}}K(k(\omega,\alpha)),
	\ee
	is smooth in terms of the pair $(\alpha,\omega)$. Here, $K=K(k)$ indicates the complete elliptic integral of first kind (see \cite{Byrd}). We have that $\alpha$ makes the role of $\alpha_3$ and to be more precisely, $\alpha$ is the square of the initial condition $\phi(0)$ and it has an intrinsic relation with the constant $B$ according with the equality $(\ref{quadratura-351})$. For a fixed $\omega_0$ and $B$ varying in the interval $(B_{\omega_0},0)$ where $B_{\omega_0}:=\frac{1-(4\omega_0+1)^{\frac{3}{2}}+6\omega_0}{12}<0$, we obtain that $\Psi(\cdot,\omega_0)$ is strictly increasing in terms of the constant $B\in (B_{\omega_0},0)$ and this fact is crucial to obtain that $\mathcal{L}$ defined in $(\ref{L-35i})$ and restricted to the space of even periodic functions has only one negative eigenvalue which is simple and zero is a simple eigenvalue associated to the eigenfunction $(0,\phi)$. To do so, we use the approaches in \cite{natali1} and \cite{neves} which give a precise information of the non-positive spectrum concerning the linear operators in the main diagonal of $\mathcal{L}$ in $(\ref{L-35i})$.\\
	\indent Finally, to obtain $(\ref{d2})$ we need to use some computations concerning the smooth curve $\omega \in I\mapsto \phi$ of periodic waves solutions with fixed period. In our analysis, it is crucial to know the derivative of the constant $B$ in $(\ref{quadratura-351})$ in terms of $\omega$. An important characteristic in our work is that we do not use numerical plots as in \cite{JaimeFabio1} and \cite{hss} to justify, for a fixed period $L>0$, that $d''(\omega)>0$ for all $\omega\in\left(\frac{1}{8}\left(\frac{\sqrt{L^2+16\pi}}{L}+\frac{8\pi}{L^2}-1\right),+\infty\right)$. Summarizing all facts described above, we are enabled to enunciate our main theorem:
	\begin{teo}[Orbital stability of the dnoidal waves for the CQNLS equation]\label{teo-nsl}
		Let $L>0$ be fixed. If $\omega\in\left(\frac{1}{8}\left(\frac{\sqrt{L^2+16\pi}}{L}+\frac{8\pi}{L^2}-1\right),+\infty\right)$, where  $\phi$ is the dnoidal  solution given in $(\ref{sol21})$, then the standing wave $u(x,t)=e^{i\omega t}\phi(x)$
		is orbitally stable in  $\mathbb{H}^1_{per,e}$.
	\end{teo}
	
	\textbf{Literature Overview.} For the case $a=1$ and $b=0$, equation in $(\ref{eq2})$ is the well known cubic nonlinear Schr\"odinger equation and it admits positive and sign changing periodic waves. For positive waves and using the approach in \cite{bona} and \cite{Weinstein}, the author in \cite{angulo} established the orbital stability of periodic standing waves solutions with dnoidal profile very similar to $(\ref{sol21})$ with $\beta=0$ (see also \cite{gallay1} and \cite{lecoz}). For periodic sign changing waves with cnoidal profile, we have the following contributions: existence of a smooth curve of periodic waves depending on $\omega$ with cnoidal profile of the form $\phi(x)=\alpha cn(bx,k)$ has also been reported in \cite{angulo}. By using the techniques introduced in \cite{grillakis2} and \cite{GSS}, the cnoidal waves were shown to be stable  in \cite{gallay1} and \cite{gallay2} with respect to anti-periodic perturbations. Spectral stability with respect to bounded or localized perturbations were also reported in \cite{gallay1}. For $\omega$ in some interval $ (0,\omega_1)$, the authors in \cite{lecoz} have established spectral stability results for the cnoidal waves with respect to perturbations with the same period  $L$ and orbital stability results in the space  constituted by anti-periodic functions with period $L/2$. The orbital stability of periodic cnoidal waves was determined in \cite{LNMP} in the same interval $(0,\omega_1)$ as above. However, the authors have restricted the analysis over the Sobolev space $\mathbb{H}_{per}^1$ constituted by zero mean periodic functions.\\
	\indent For the case $a=0$ and $b=1$, we have the work \cite{JaimeFabio1} where, for a fixed $L>0$, the authors showed the existence of a unique $\omega_2>\frac{\pi^2}{L^2}$ such that the periodic wave with dnoidal profile very similar to $(\ref{sol21})$ is orbitally stable for all $\omega\in\left(\frac{\pi^2}{L^2},\omega_2\right)$ and orbitally unstable for all $\omega\in (\omega_2,+\infty)$. A different value $\omega_3$ compared with $\omega_2$ in \cite{JaimeFabio1} has been reported in \cite{hss} for the same equation to prove the spectral stability/instability of periodic waves with a similar profile as in $(\ref{sol21})$. The discrepancy between $\omega_2$ and $\omega_3$ can have been occasioned by a numerical error in the computational approach. However, this fact is irrelevant since in both cases the authors obtained "stability" for small values of $\omega$ and "instability" for large ones.\\
	\indent In the case $a=b=1$, the author in \cite{cesar} established the orbital stability of periodic standing waves of the form $(\ref{PW2})$ in the whole energy space $\mathbb{H}_{per}^1$ by using the classical approaches \cite{grillakis2}-\cite{GSS}. The spectral analysis was borrowed of \cite{natali2} and to calculate $d''(\omega)$ the author employed some numerical computations with Maple program to plot the graphic of $d''(\omega)$ in terms of the modulus $k\in (0,1)$ and $L>0$. Our method is different since we use the monotonicity of the period map in $(\ref{period-map1})$ in terms of the initial data $\alpha$ to obtain the spectral analysis and the evaluation for $d''(\omega)$ seems more clear since we use an analytical argument.
	\begin{obs}\label{obs12} It is important to highlight that the abstract theories in \cite{grillakis2}-\cite{GSS} can not be directly applied to conclude the orbital stability in the whole energy space $\mathbb{H}_{per}^1$ when the simple standing wave of the form $(\ref{PW2})$ is considered for the equation $(\ref{eq2})$. In fact, the abstract approach \cite{grillakis2} can be applied to deduce the orbital stability in the whole energy space $\mathbb{H}_{per}^1$ when it is considered a periodic wave containing both translation and rotation symmetries while \cite{GSS} only works by considering one symmetry only. Following the arguments in \cite{Weinstein} (see also \cite{angulo} and \cite{natali1}), it is possible to consider standing waves of the form $(\ref{PW2})$ to prove the orbital stability in $\mathbb{H}_{per}^1$ by considering the two basic symmetries. However and according to our best knowledge, global solutions in time are needed to this end. As far as we know, local solutions for the initial value problem associated to $(\ref{eq21})$ can be determined (see \cite[Chapter 5]{Iorio}) while global solutions are not expected because of the presence of the critical power nonlinearity $|u|^4u$ which gives a blow up phenomena in finite time in $\mathbb{H}_{per}^1$. Thus, the correct space to use the approach \cite{GSS} by considering only local solutions is $\mathbb{H}_{per,e}^1$, not the whole space $\mathbb{H}_{per}^1$.\end{obs}
	
	\indent To finish our revision of literature: by considering several cases of $a,b\in\mathbb{R}$ and more general double-power nonlinearities in equation $(\ref{eq2})$, the orbital stability/instability of solitary standing waves has been studied in details in \cite{Ohta}. To do so, the author employed the fact that explicit solutions depending on the hyperbolic functions can be determined. This information is very useful to apply the classical Sturm-Liouville theory to obtain the spectral information of the non-positive spectrum of the associated linear operator $(\ref{L-35i})$ in the infinite wavelength scenario. After that, the stability/instability can be determined by a direct application of \cite{grillakis2} and \cite{GSS} by calculating the value of $d''(\omega)$ for each solitary wave. In particular, for the case $a=b=1$, the explicit solitary wave solution depending on $\omega$ and given by
	\begin{eqnarray}\label{solitw1}
		\phi(x)=\sqrt{\frac{12\omega}{3+\sqrt{48\omega+9}\ cosh(2\sqrt{\omega}x)}},\ \ \ \ \ \ \omega>0.
	\end{eqnarray}	
	is orbitally stable.\\
	\indent Our paper is organized as follows: Section 2 has some basic notations used in paper. In Section 3, we present the existence of periodic waves with fixed period. Section 4 is devoted to the spectral analysis of the linear operator $\mathcal{L}$ in $(\ref{L-35i})$. Finally in Section 5, we present the orbital stability/instability result.

	\section{Notation}\label{sec2}
	Here we introduce the basic notation concerning the periodic Sobolev spaces. For a more complete introduction to these spaces we refer the reader to \cite{Iorio}. By $L^2_{per}:=L^2_{per}([0,L])$, $L>0$, we denote the space of all square integrable functions which are $L$-periodic. For $s\geq0$, the Sobolev space
	$H^s_{per}:=H^s_{per}([0,L])$
	is the set of all periodic distributions such that
	$$
	\|f\|^2_{H^s_{per}}:= L \sum_{k=-\infty}^{\infty}(1+|k|^2)^s|\hat{f}(k)|^2 <\infty,
	$$
	where $\hat{f}$ is the periodic Fourier transform of $f$. The space $H^s_{per}$ is a  Hilbert space with natural inner product denoted by $(\cdot, \cdot)_{H^s}$. When $s=0$, the space $H^s_{per}$ is isometrically isomorphic to the space  $L^2_{per}$, that is, $L^2_{per}=H^0_{per}$ (see, e.g., \cite{Iorio}). The norm and inner product in $L^2_{per}$ will be denoted by $\|\cdot \|_{L^2}$ and $(\cdot, \cdot)_{L^2}$.
	
	In addition, to simplify notation we set
	$$\mathbb{H}^s_{per}:= H^s_{per} \times H^s_{per},\ \
	\mathbb{L}^2_{per}:= L^2_{per} \times L^2_{per},$$
	endowed with their usual norms and scalar products. When necessary and since $\mathbb{C}$ can be identified with $\mathbb{R}^2$, notations above can also be used in the complex/vectorial case in the following sense: for $f\in \mathbb{H}_{per}^s$ we have $f=f_1+if_2\equiv (f_1,f_2)$, where $f_i\in H_{per}^s$ $i=1,2$. For $s\geq0$, the  space $H_{per,e}^s$ is the Sobolev space $H_{per}^s$ constituted by even periodic functions. Similarly, we can define $\mathbb{H}_{per,e}^s$.\\

	\section{Existence of periodic waves}\label{sec3}
	\subsection{Periodic waves with fixed period.} This section is devoted to establish the existence of a smooth curve of periodic solutions for the equation (\ref{eq2}). To do so, we use the quadrature method and similar arguments as in \cite{JaimeFabio1} and used thereafter in \cite{cesar} to construct periodic waves of the form $(\ref{sol21})$ for the same model $(\ref{eq21})$. However, it is important to mention that we fill in some gaps left in the construction of periodic waves in \cite{cesar} since some points in the proof are not totally clear. We see that $(\ref{edo})$ can be expressed in a quadrature form as
	\be
	(\phi')^2=-\frac{\phi^6}{3}-\frac{\phi^4}{2}+\omega\phi^2+B, \label{quadratura-35}
	\ee
	where $B$ is an integration constant.
	
	In order to obtain positive and periodic solutions, let us assume that $\psi=\phi^2$. We obtain a new quadrature form in terms of $\psi$ given by
	\begin{equation}\label{pol1}
		\left(\psi'\right)^2=\di-\frac{4}{3}\psi^4-2\psi^3+4\omega\psi^2+4B\psi=\di\frac{4}{3}P(\psi),
	\end{equation}
	where $P(s)$ indicates the quartic polynomial given by $P(s)=-s^4-\frac{3s^3}{2}+3\omega s^2+3Bs$. 
	
	By $(\ref{pol1})$, we see that $0$ is a root of $P$. To obtain positive and periodic explicit solutions, we assume that the roots of $P$, named as $\alpha_1,\alpha_2, \alpha_3$, are real numbers and satisfying $\alpha_1<0<\alpha_2<\alpha_3$. Thus, $P$ can be factorized as $p(s)=s(s-\alpha_1)(s-\alpha_2)(\alpha_3-s)$, so that 
	\begin{equation}\label{eq1}
		\left(\psi'\right)^2=\frac{4}{3}\psi(\psi-\alpha_1)(\psi-\alpha_2)(\alpha_3-\psi).
	\end{equation}
	\indent Since $\psi>0$, we can consider 
	$\alpha_1<0<\alpha_2<\psi<\alpha_3 $. All the roots $\alpha_i$ need to satisfy, by $(\ref{pol1})$ and $(\ref{eq1})$, the following system of nonlinear equations
	\begin{equation}\label{sis1}
		\left\{ \begin{array}{rcl}
			\alpha_1+\alpha_2+\alpha_3&=&-\frac{3}{2}\\
			\alpha_1\alpha_2+\alpha_1\alpha_3+\alpha_2\alpha_3&=&-3\omega \\
			\alpha_1\alpha_2\alpha_3&=&3B.
		\end{array}
		\right.
	\end{equation}
	
	Using \cite[Formula 257.00]{Byrd} in the quadrature form (\ref{eq1}), we obtain the following explicit periodic solution depending on the Jacobi elliptic function of \textit{dnoidal} type as
	$$
	\psi(x)=\frac{\alpha_3dn^2\left(\frac{2}{\sqrt{3}g}x,k\right)}{1+\beta^2sn^2\left(\frac{2}{\sqrt{3}g} x,k\right)},
	$$
	where $sn$ indicates the elliptic function of $snoidal$ type. Thus, since $\psi=\phi^2$, it follows that
	\begin{equation}
		\phi(x)= \frac{\sqrt{\alpha_3}dn\left(\frac{2}{\sqrt{3}g}x,k\right)}{\sqrt{1+\beta^2sn^2\left(\frac{2}{\sqrt{3}g} x,k\right)}} \label{sol2}
	\end{equation}
	where parameters $g$ and $k$ are expressed by
	\be\label{kg} g=\frac{2}{\sqrt{\alpha_3(\alpha_2-\alpha_1)}}, \quad k^2=\frac{-\alpha_1(\alpha_3-\alpha_2)}{\alpha_3(\alpha_2-\alpha_1)},\ee
	and $\beta^2=-\frac{\alpha_3}{\alpha_1}k^2>0$.
	
	Now, since the dnoidal function has fundamental period $2K$, it follows that $\phi$  in (\ref{sol2}) has fundamental period $\mathcal{T}_{\phi}$ given by
	\be \mathcal{T}_{\phi}(\alpha_1,\alpha_2,\alpha_3,k)=\frac{2\sqrt{3}K(k)}{\sqrt{\alpha_3(\alpha_2-\alpha_1)}}.\label{T}\ee
	By $(\ref{sis1})$, we can obtain $\alpha_1$ and $\alpha_2$ in terms of $\alpha_3$ as
	\be \alpha_1=-\frac{\sqrt{3}\, \sqrt{q(\alpha_3)}+2 {\alpha_3}+3}{4},\, \alpha_2=\frac{\sqrt{3}\, \sqrt{q(\alpha_3)}-2 {\alpha_3}-3}{4},\label{alp}\ee
	where $q$ is polynomial defined by  $q(\alpha_3)=16\omega-4\alpha_3^2-4\alpha_3+3$ which we would like to conclude that it is positive to make sense the expressions in $(\ref{alp})$. In fact, function $ \alpha_2 (\alpha_3) $ defined in (\ref {alp}) is strictly decreasing in terms of $\alpha_3$ and its maximum value occurs at the point $\alpha_3=\frac{\sqrt{1+4\omega}-1}{2}$. Since
	$0<\alpha_2<\alpha_3$, we have
	$$0<\alpha_2<\frac{\sqrt{1+4 \omega}-1}{2}<\alpha_3<\frac{\sqrt{48\omega+9}-3}{4}.$$ It is possible to determine the roots of $q$ in terms of $\omega$ and they are given by $x_{\omega}=-\frac{2 \sqrt{1+4 \omega}+1}{2}$ and 
	$y_{\omega}=\frac{2 \sqrt{1+4 \omega}-1}{2}$. Since $0<\alpha_3<\frac{\sqrt{48\omega+9}-3}{4}<\frac{2 \sqrt{1+4 \omega}-1}{2}$, we obtain that $q(\alpha_3)>0$ as desired. Next, using the expressions $\alpha_1$ and $\alpha_2$ in $(\ref{alp})$, we obtain by $(\ref{sis1})$ that $B$ is negative and it can be given in terms of $\alpha_3$ and $\omega$ as
	\be\label{B}
	B=-\alpha_3\omega+\frac{\alpha_3^3}{3}+\frac{\alpha_3^2}{2}.
	\ee

	We shall give some asymptotic behaviours concerning the period map $\mathcal{T}_{\phi}$ in $(\ref{T})$ and the periodic wave in (\ref{sol2}).
	First, we see that
	\be \mathcal{T}_{\phi}> \frac{2\pi}{\sqrt[4]{4\omega+1}\sqrt{\sqrt{4\omega+1}-1}}.\label{tl}\ee
	In fact, we have that $\alpha_3\longrightarrow\frac{\sqrt{4 \omega+1}-1}{2}$ implies by $(\ref{alp})$ that $\alpha_2\longrightarrow\frac{\sqrt{4 \omega+1}-1}{2}$ and $\alpha_1\rightarrow -\frac{1+2\sqrt{1+4\omega}}{2}$. Since $k^2=\frac{-\alpha_1(\alpha_3-\alpha_2)}{\alpha_3(\alpha_2-\alpha_1)}$, we obtain $k\longrightarrow 0^{+}$, $K(k)\longrightarrow \pi/2^+$ and by $(\ref{T})$, we have $\mathcal{T}_{\phi}\longrightarrow \frac{2\pi}{\sqrt[4]{4\omega+1}\sqrt{\sqrt{4\omega+1}-1}}$. This last fact gives us by (\ref{sol2}) and the facts $dn(u,0^+)\sim 1$ and $sn(u,0^+)\sim sin(u)$ that
	$
	\phi(x)=\sqrt{\frac{\sqrt{4\omega+1}-1}{2}}
	$ is the equilibrium solution for $(\ref{edo})$. On the other hand, $\alpha_3\longrightarrow\frac{\sqrt{48 \omega+9}-3}{4}$ implies $\alpha_2\longrightarrow 0$ and $\alpha_1\longrightarrow -\frac{\sqrt{48\omega+9}+3}{4}$. Since in this case we have $k\longrightarrow 1^-$ and $K(k)\longrightarrow +\infty$, it follows that $\mathcal{T}_{\phi}\longrightarrow +\infty$. Next, the period-map $\mathcal{T}_{\phi}$ is strictly increasing in terms of $\alpha_3$ (see Theorem $\ref{teo1}$) and we obtain the estimate $(\ref{tl})$ as required.\\
	\indent The elliptic functions of dnoidal and snoidal type have an asymptotic behaviour in this special case of infinite wavelength scenario. In fact, since with $dn(u,1^-)\sim sech(u)$ and $sn(u,1^-)\sim tanh(u)$, we obtain that our periodic waves converge exactly to the solitary wave as reported in \cite{Ohta}  and given by
	\be\label{solitw}
	\phi(x)=\sqrt{\frac{12\omega}{3+\sqrt{48\omega+9}\ cosh(2\sqrt{\omega}x)}},\ \ \ \ \ \ \omega>0.
	\ee	
	\indent It is important to notice that the modulus $k$ $(\ref{kg})$ and the fundamental period $\mathcal{T}_{\phi}$ in $(\ref{T})$ can both be seen as functions of $\alpha_3$ and $\omega$. In fact, by (\ref{alp}) we have
	\be
	\mathcal{T}_{\phi}(\alpha_3,\omega)=\frac{\sqrt{8} {{3}^{\frac{1}{4}}} K(k) }{\sqrt{{\alpha_3}} {{\left( q(\alpha_3)\right) }^{\frac{1}{4}}}},\,\, k^2(\alpha_3,\omega)=\frac{\sqrt{3}\, {\alpha_3} \sqrt{q(\alpha_3)}-12 \omega+6 {{{\alpha_3}}^{2}}+9 {\alpha_3}}{2 \sqrt{3}\, {\alpha_3} \sqrt{q(\alpha_3)}}. \label{ka}
	\ee

	Analysis above allows us to obtain a dnoidal wave solution for equation (\ref{eq2}) with a fixed period $L > 0$.  Indeed, let $L>0$ be fixed and consider 
	$$\omega>\di\frac{1}{8}\left(\frac{\sqrt{L^2+16\pi}}{L}+\frac{8\pi}{L^2}-1\right).$$ 
	Since the map $\alpha\in(\frac{\sqrt{1+4 \omega}-1}{2},\frac{\sqrt{48\omega+9}-3}{4}) \mapsto \mathcal{T}_{\phi}(\alpha,\omega)$ is strictly increasing in terms of $\alpha$, we see by the implicit function theorem that there exists a unique  $\alpha_3=\alpha_3(\omega)\in(\frac{\sqrt{1+4 \omega}-1}{2},\frac{\sqrt{48\omega+9}-3}{4})$, such that the fundamental period of the dnoidal wave (\ref{sol2}) is $\mathcal{T}_{\phi}(\alpha_3(\omega),\omega)=L$ (for the proof, see Theorem \ref{teo1} below).

	Let $L>0$ be fixed. In the next result, we show the existence of a smooth curve of $L-$periodic dnoidal waves solutions for the equation (\ref{edo}) depending smoothly on $\omega$.
	
	\begin{teo}\label{teo1} Let $L>0$  be fixed. For $\omega>\frac{1}{8}\left(\frac{\sqrt{L^2+16\pi}}{L}+\frac{8\pi}{L^2}-1\right)$, consider the unique $\alpha_{3,0}=\alpha_3(\omega_0)\in\left(\frac{ \sqrt{1+4\omega_0}-1}{2},\frac{\sqrt{48\omega_0+9}-3}{4}\right)$ such that $\mathcal{T}_{\phi}(\alpha_{3,0},\omega_0)=L$. Then,
		
		\begin{enumerate}
			\item there exist an interval $I_1$ around $\omega_0$, an interval $I_2$ around $\alpha_{3,0}$ and a unique 	smooth function $\Lambda:I_1\mapsto I_2$, such that $\Lambda(\omega_0)=\alpha_{3,0}$ such that
			\be
			\mathcal{T}_{\phi}(\alpha_3(\omega),\omega)=\frac{\sqrt{8} {{3}^{\frac{1}{4}}} K(k) }{\sqrt{{\alpha_3}} {{\left( q(\alpha_3)\right) }^{\frac{1}{4}}}}=L, \label{l1}
			\ee
			where $\omega\in I_1$, $\alpha_3(\omega):=\Lambda(\omega)\in I_2$ and $k^2=k^2(\omega)\in (0,1)$ is defined by (\ref{ka}).
			\item The dnoidal wave solution in (\ref{sol2}), $\phi_{\omega}:=\phi_{\omega}(\cdot, \alpha_3(\omega))$, determined by $\alpha_3(\omega)$, has
			fundamental period L and satisfies (\ref{edo}). Moreover, the mapping
			\be
			\omega\in I_1\mapsto \phi_{\omega}\in H_{per}^n([0,L])
			\ee
			is a smooth function for all $n\in\mathbb{N}.$
			\item The interval $I_1$ can be chosen as $\left(\frac{\sqrt{L^2+16\pi}}{8L}+\frac{\pi}{L^2}-\frac{1}{8},+\infty\right)$.
		\end{enumerate}	
	\end{teo}	
	\begin{proof}
		We apply the implicit function theorem by following similar ideas as in \cite{JaimeFabio1}. First, let us consider the open set \begin{center}
			$\Omega=\left\{(\alpha,\omega);\ \omega>0,\ \alpha\in\left(\frac{ \sqrt{1+4\omega}-1}{2},\frac{\sqrt{48\omega+9}-3}{4}\right) \right\}\subset \R^2$
		\end{center}
		and define the so-called \textit{period map} $\Psi:\Omega\rightarrow\R$ by
		\be\label{period-map}
		\Psi(\alpha,\omega)=\frac{\sqrt{8} {{3}^{\frac{1}{4}}}  }{\sqrt{{\alpha}} {{\left( q(\alpha)\right) }^{\frac{1}{4}}}}K(k(\omega,\alpha)),
		\ee
		where $k^2$ is defined in (\ref{ka}) and by hypothesis, we have $\Psi(\alpha_{3,0},\omega_0)=L$. We show that $	\frac{\partial \Psi}{\partial \alpha}(\alpha,\omega)>0$ in $\Omega$. In fact, we have that
		\begin{equation}\label{dP}
			\begin{array}{lllll}\displaystyle\frac{\partial \Psi}{\partial \alpha}(\alpha,\omega)&=&\displaystyle\frac{\sqrt{8} {3}^{\frac{1}{4}}}{(\alpha^2 q(\alpha))^{5/4}}\left[ \alpha^2 q(\alpha) \frac{dK}{dk}\frac{dk}{d\alpha}\right.\\\\
				&-&\displaystyle\left.\alpha K(k)\left(\frac{16 \omega-8 {{{\alpha}}^{2}}-6 {\alpha}+3}{2}\right) \right]. \end{array}
		\end{equation}
		On the other hand, 
		\be \frac{dk}{d\alpha}=\frac{1}{2k}\frac{12 {{\alpha}^{3}}+18 {{\alpha}^{2}}-36 \omega \alpha+96 {{\omega}^{2}}+18 \omega}{\sqrt{3}\, {{q(\alpha)}^{\frac{3}{2}}}\, {{\alpha}^{2}}}=\frac{r(\alpha)}{k\sqrt{3}\, {{q(\alpha)}^{\frac{3}{2}}}\, {{\alpha}^{2}}},\label{dk}\ee
		where $r(\alpha)=6 {{\alpha}^{3}}+9 {{\alpha}^{2}}-18 \omega \alpha+48 {{\omega}^{2}} +9 \omega$. Since $\alpha>0$, we obtain  $r(\alpha)>0$ and thus, $\frac{dk}{d\alpha}>0$.
		
		By  (\ref{dP}) and (\ref{dk}), we can write
		\begin{eqnarray*}
			\frac{(\alpha^2 q(\alpha))^{5/4}}{\sqrt{8} {3}^{\frac{1}{4}}}\frac{\partial \Psi}{\partial \alpha}(\alpha,\omega)&=&  \alpha^2 q(\alpha) \frac{dK}{dk}\frac{dk}{d\alpha}-\alpha K(k)\left(\frac{16 \omega-8 {{{\alpha}}^{2}}-6 {\alpha}+3}{2}\right) \\
			&=&  \alpha^2 q(\alpha) \frac{dK}{dk}\frac{r(\alpha)}{k\sqrt{3}\, {{q(\alpha)}^{\frac{3}{2}}}\, {{\alpha}^{2}}}-\alpha K(k)\left(\frac{16 \omega-8 {{{\alpha}}^{2}}-6 {\alpha}+3}{2}\right)\\
			&=&  \frac{dK}{dk}\frac{r(\alpha)}{k\sqrt{3}\, \sqrt{q(\alpha)}\, }-\alpha K(k)\left(\frac{16 \omega-8 {{{\alpha}}^{2}}-6 {\alpha}+3}{2}\right)\\
			&=&\frac{r(\alpha)}{\sqrt{3}\, \sqrt{q(\alpha)}\, } \left( \frac{dK}{dk}\frac{1}{k}-K(k)\frac{\alpha\sqrt{3q(\alpha)}(16 \omega-8 {{{\alpha}}^{2}}-6 {\alpha}+3)}{2r(\alpha)}\right)\\
			&>&\frac{r(\alpha)}{\sqrt{3}\sqrt{q(\alpha)}}\left(\frac{dK}{dk}\frac{1}{k}-K(k) \right),
		\end{eqnarray*}
		where we are using the fact that $\frac{\alpha\sqrt{3q(\alpha)}(16 \omega-8 {{{\alpha}}^{2}}-6 {\alpha}+3)}{2r(\alpha)}<1$ over the set $\Omega$. Since $\frac{dK}{dk}\frac{1}{k}-K(k)>0$, one has $\frac{\partial\Psi}{\partial\alpha}>0$ for every $(\alpha,\omega)\in \Omega.$ By the implicit function theorem, there exists an interval $I_1$ around $\omega_0$, an interval $I_2$ around $\alpha_{3,0}$ and a unique smooth function $\Lambda:I_1\longrightarrow I_2$ such that $\Psi(\Lambda(\omega),\omega)=L$ for every $\omega\in I$. This fact allows to conclude the first two items in theorem. 
		
		Since $\omega$ was chosen arbitrarily in the interval $I_1$, it follows from the uniqueness of the function $\Lambda$ that $I_1$ can be extended to $\left(\frac{\sqrt{L^2+16\pi}}{8L}+\frac{\pi}{L^2}-\frac{1}{8},+\infty\right)$. This completes
		the proof of the theorem.
	\end{proof}	
	
	\begin{coro}\label{coro1} Let $\Lambda: I_1 \rightarrow I_2$ be given in Theorem \ref{teo1}. Thus, $\alpha_3(\omega)=\Lambda(\omega)$ is a strictly increasing function in $I_1$. Moreover, the modulus function given by
		$$k^2(\omega)=\frac{\sqrt{3}\, {\alpha_3} \sqrt{q(\alpha_3)}-12 \omega+6 {{{\alpha_3}}^{2}}+9 {\alpha_3}}{2 \sqrt{3}\, {\alpha_3} \sqrt{q(\alpha_3)}},$$
		is a strictly decreasing function.
	\end{coro}
	\begin{proof} Using the proof of  Theorem \ref{teo1},  we can differentiating the equality  $\Psi(\Lambda(\omega),\omega)=L$ in terms of $\omega$ to obtain
		\be\label{dgo}
		\frac{d\Lambda}{d\omega}=\di-\frac{\frac{\partial \Psi}{\partial \omega}}{\frac{\partial \Psi}{\partial\alpha}}.
		\ee
		Since $\frac{\partial \Psi}{\partial\alpha}>0$, we only need to calculate $\frac{\partial \Psi}{\partial \omega}$. In fact, we have by $(\ref{period-map})$ that
		\begin{equation}\label{dLw}
			\di\frac{\partial\Psi}{d\omega}=\di\frac{\sqrt{8} {3}^{\frac{1}{4}}}{(\alpha^2 q(\alpha))^{5/4}}\left[  \alpha^2q(\alpha) \frac{dK}{dk}\frac{dk}{d\omega}-4\alpha^2K(k)  \right]. 
		\end{equation}
		\indent We need to establish a convenient expression for $\frac{dk}{d\omega}$ in $(\ref{dLw})$. In fact,  we get
		\begin{equation}\label{dkdomega}
			\frac{dk}{d\omega}=-\frac{\sqrt{3}}{k\alpha q(\alpha)^{\frac{3}{2}}}(2\alpha+8\omega+3)<0.
		\end{equation}
		Combining $(\ref{dgo})$, $(\ref{dLw})$ and $(\ref{dkdomega})$ we obtain $\frac{d\Lambda}{d\omega}>0$. This completes the proof.
	\end{proof}
	
	\begin{obs}\label{rem1}
		Using the relations in $(\ref{alp})$ and Corollary $\ref{coro1}$, it follows that $\alpha_1'(\omega)>0$ and $\alpha_2'(\omega)<0$ for all $\omega\in I_1$. Moreover, it is possible to obtain the value of $B$ in terms of $\alpha_1$ and $\omega$ as
		\be\label{Balpha1}
		B=-\alpha_1\omega+\frac{\alpha_1^3}{3}+\frac{\alpha_1^2}{2}.
		\ee
		
	\end{obs}
	
	\subsection{Characterization of positive and periodic waves.}\label{sec3-2} Next, let $\omega_0>0$ be fixed. In what follows, let us define the period map in Theorem $\ref{teo1}$ in terms of $\alpha\in \left(\frac{ \sqrt{1+4\omega_0}-1}{2},\frac{\sqrt{48\omega_0+9}-3}{4}\right):=I_3$, by $T(\alpha):=\Psi(\alpha,\omega_0)$. We see that 
	\be\label{perT}T:I_3\longrightarrow\left( \frac{2\pi}{\sqrt[4]{4\omega_0+1}\sqrt{\sqrt{4\omega_0+1}-1}},+\infty\right),\ee
	is a smooth strictly increasing function according to the proof of Theorem $\ref{teo1}$. In addition, using the standard ODE theory, we obtain for $\alpha \in I_3$ that all positive even periodic waves satisfy the IVP
	\begin{equation}\label{IVP}\left\{\begin{array}{lllll}
			-\phi''+\omega_0\phi-\phi^3-\phi^5=0,\\
			\phi(0)=\sqrt{\alpha},\\
			\phi'(0)=0.\end{array}\right.\end{equation}
	
	By $(\ref{quadratura-35})$, we have at $x=0$ that $\alpha$ and $B$ must satisfy the equality \be\label{eqalpha}-\frac{\alpha^3}{3}-\frac{\alpha^2}{2}+\omega_0\alpha+B=0,\ee
	so that $\alpha:\left(\frac{1-(4\omega_0+1)^{\frac{3}{2}}+6\omega_0}{12},0\right)\longrightarrow I_3$ can be seen as a smooth function depending on $B$. Let us consider $B_{\omega_0}:=  \frac{1-(4\omega_0+1)^{\frac{3}{2}}+6\omega_0}{12}<0$. By the chain rule, we deduce that the period map $T$ is also smooth in terms of $B\in (B_{\omega_0},0)$. In addition, we have 
	\be\label{permap1}
	\frac{dT}{dB}=\frac{dT}{d\alpha}\frac{d\alpha}{dB}.
	\ee
	\indent We need to know the monotonicity behaviour of $\frac{dT}{dB}$. Indeed, differentiating the initial condition $\phi(0)=\sqrt{\alpha}$ in $(\ref{IVP})$ with respect to $B$, we obtain $\frac{d\alpha}{dB}=2\phi(0)\frac{d\phi}{dB}(0)$. This last fact enables to differentiate equality $(\ref{eqalpha})$ in terms of $B$ to get
	\be\label{deralpha}
	(-\phi(0)^5-\phi(0)^3+\omega_0\phi(0)^2)\frac{d\phi}{dB}(0)=-1.
	\ee
	Since $\phi''(0)= -\phi(0)^5-\phi(0)^3+\omega_0\phi(0)^2$ by $(\ref{edo})$, we obtain from the fact $\phi(0)=\sqrt{\alpha}$ is the maximum point of the periodic wave $\phi$ that $\phi''(0)<0$. This last fact allows to conclude by $(\ref{deralpha})$ that $\frac{d\alpha}{dB}>0$. Thus, using $(\ref{permap1})$ joint with the fact $\frac{dT}{d\alpha}>0$ (see Theorem $\ref{teo1}$), we obtain $\frac{dT}{dB}>0$ for all $B\in (B_{\omega_0},0)$. Summarizing the arguments above, we have
	\begin{prop}\label{prop2} Let $\omega_0>0$ be fixed. The period map $$T:(B_{\omega_0},0)\longrightarrow \left( \frac{2\pi}{\sqrt[4]{4\omega_0+1}\sqrt{\sqrt{4\omega_0+1}-1}},+\infty\right),$$ is a smooth function in terms of $B\in(B_{\omega_0},0)$  and $\frac{dT}{dB}>0$.
	\end{prop}
	\begin{flushright}
		$\square$
	\end{flushright}

	\indent Combining the results in Theorem \ref{teo1} and Proposition \ref{prop2}, we obtain the following result:
	
	\begin{teo}\label{teo-uniq}
		Let $L>0$ be fixed and consider $\omega>\frac{1}{8}\left(\frac{\sqrt{L^2+16\pi}}{L}+\frac{8\pi}{L^2}-1\right)$. All positive and even periodic solution associated to the equation $(\ref{edo})$ for the case $a=b=1$ has the dnoidal profile in $(\ref{sol2})$.
	\end{teo}
	
	\begin{proof}
		Under the assumptions of the proposition, we obtain by the monotonicity of the period map $T:(B_{\omega_0},0)\longrightarrow \left( \frac{2\pi}{\sqrt[4]{4\omega_0+1}\sqrt{\sqrt{4\omega_0+1}-1}},+\infty\right)$ the existence of a unique $B_0\in (B_{\omega_0},0)$ such that $T(B_0)=L$. This unique $B_0$ corresponds to a unique $\alpha_0:=\alpha_{3,0}\in I_3$ by using $(\ref{eqalpha})$ and the fact that $\frac{d\alpha}{dB}>0$ for all $B\in(B_{\omega_0},0)$. The result is then established by a direct application of Theorem $\ref{teo1}$.
	\end{proof}
	
	\section{Spectral Properties.}\label{sec4}
	\subsection{Floquet theory framework.}\label{sec4-1}
	Our intention is to recall some basic facts about Floquet's theory (for further details, see \cite{est} and \cite{magnus}). In what follows, let $Q$ be a smooth $T$-periodic function.  Consider $\mathcal{P}$ the Hill operator defined in $L_{per}^2([0,T])$, with domain $D(\mathcal{P})=H_{per}^2([0,T])$ as
	$$
	\mathcal{P}=-\partial_x^2+Q(x).
	$$
	It is well known that the spectrum of $\mathcal{P}$  is formed by an unbounded sequence of
	real eigenvalues
	\be\label{seq}
	\lambda_0 < \lambda_1 \leq \lambda_2 < \lambda_3 \leq \lambda_4<
	\cdots\; < \lambda_{2n-1} \leq \lambda_{2n}\; \cdots,
	\ee
	where equality means that $\lambda_{2n-1} = \lambda_{2n}$  is a
	double eigenvalue. In addition, according with the classical Oscillation Theorem, we see that the spectrum of $\mathcal{P}$ can be characterized by the number of zeros
	of the associated eigenfunctions. In fact, if $\varphi$ is an eigenfunction associated to either $\lambda_{2n-1}$ or $\lambda_{2n}$, then $\varphi$ has exactly
	$2n$ zeros in the half-open
	interval $[0, T)$. In particular, the even eigenfunction associated to the first eigenvalue $\lambda_0$ has no zeros in $[0, T]$.
	
	Let $\varphi$ be a non-trivial $T$-periodic solution of the Hill equation
	\begin{equation}\label{zeqL}
		-f''+Q(x)f=0.
	\end{equation}
	Consider $y$ the another solution of \eqref{zeqL} which is linearly independent with $\varphi$.  There exists a constant $\theta$ (depending on $y$ and $\varphi$) such that
	\be\label{theta1}
	y(x+T)=y(x)+\theta \varphi(x).
	\ee
	Thus, we see that $\theta=0$ is a necessary and sufficient condition to become $y$ as a periodic solution for $(\ref{zeqL}).$
	
	Next result gives that it is possible to decide the exact position of the zero eigenvalue by knowing the precise sign of $\theta$ in $(\ref{theta1})$.
	
	\begin{lema}\label{specprop}
		Let $\theta$ be the constant given by (\ref{theta1}) and suppose that $\varphi$ is an $T-$periodic solution for the equation $(\ref{zeqL})$ containing only two zeros over $[0,T).$ The
		eigenvalue $\lambda=0$ is simple if and only if $ \theta \neq 0$.
		Moreover, if $\theta \neq 0$, then  $ \lambda_{1}=0$ if $\theta <
		0$, and $ \lambda_{2}=0$ if $\theta > 0$.
	\end{lema}
	\begin{proof}
		See \cite{natali1} and \cite{neves}.
	\end{proof}
	
	\subsection{Spectral analysis for the operator $\mathcal{L}$.}\label{sec4-2} In this subsection, we study the linearized operator $\mathcal{L}$ defined in $(\ref{L-35i})$. First of all, we see that $\mathcal{L}$ is a diagonal operator and thus, it is possible to obtain the spectral information concerning the non-positive spectrum of $\mathcal{L}$ by analysing separately the non-positive spectrum of the following linear operators \be\mathcal{L}_1=-\partial_x^2+\omega-3\phi^2-5\phi^4,\label{op1}\ee 
	and 
	\be\mathcal{L}_2=-\partial_x^2+\omega-\phi^2-\phi^4.\label{op2}\ee 
	Concerning $\mathcal{L}_2$, we have the following result:

	\begin{lema}\label{lema14}
		Let $L>0$ be fixed and consider $\omega>\frac{1}{8}\left(\frac{\sqrt{L^2+16\pi}}{L}+\frac{8\pi}{L^2}-1\right)$. Operator
		$\mathcal{L}_{2}$
		in $(\ref{op2})$ defined in $L_{per}^2([0,L])$ with domain in
		$H_{per}^2([0,L])$ is a non-negative operator. Zero is a
		simple eigenvalue with eigenfunction $\phi$. Moreover, the remainder of the spectrum of $\mathcal{L}_{2}$
		is constituted by a discrete set
		of eigenvalues bounded away from zero.
	\end{lema}
	\begin{proof}
		By $(\ref{edo})$, we see that $\mathcal{L}_2\phi=0$. Since $\phi>0$, we obtain by the classical Floquet theory that $\lambda_0=0$ is the first eigenvalue for $\mathcal{L}_2$ which is clearly simple by $(\ref{seq})$. 
	\end{proof}
	\indent Before giving the behaviour of the non-positive spectrum of $\mathcal{L}_1$, we need some preliminary tools. First of all for the solution $\phi$ obtained by Theorem $\ref{teo-uniq}$, we have that $\phi'$ is a solution of the Hill equation $\mathcal{L}_1\phi'=0$ by deriving equation $(\ref{edo})$ with respect to $x$. In addition, the dnoidal solution in $(\ref{sol2})$ is positive and $\phi'$ is an odd function having two zeros in the interval $[0,L)$. By the classical Floquet theory, we see that $\lambda_1=0$ or $\lambda_2=0$. We show that $\lambda_1=0$ and it results to be simple.
	\begin{lema} \label{teo13}
		We have, $\frac{dT}{dB}=-\frac{\theta}{2}$, where $\theta$ is the constant in \eqref{theta1}.
	\end{lema}
	\begin{proof}
		In a general setting, consider $\{\phi',\bar{y}\}$ a fundamental set of solutions for the Hill equation $\mathcal{L}_1y=0$, where $y\in C^{\infty}([0,T])$. Thus, we have that $\bar{y}$ and $\phi'$ satisfies the equation $(\ref{zeqL})$ for $y=\bar{y}$ and $\varphi=\phi'$. In addition, the Wronskian $W$ associated to fundamental set satisfies $W(\phi'(x),\bar{y}(x))=1$ for all $x\in[0,T]$. Since $\phi'$ is odd and periodic, we obtain that $\bar{y}$ is even and it satisfies the following IVP
		
		\be \left\{
		\begin{array}{l}
			-y''+ \omega y-3\phi^2y-5\phi^4y = 0 ,\\
			y(0) =-\frac{1}{\phi''(0)}, \\
			y'(0)= 0.
		\end{array} \right.
		\label{y}
		\ee	
		The smoothness of $\phi'$ in terms of the parameter $B$ enables us to take the derivative of $\phi'(T)=0$ with respect to $B$ to obtain 
		\begin{equation}\label{eq12343}
			\phi''(T)\frac{dT}{dB}+\frac{\partial \phi'}{\partial B}(T)=0.\end{equation}
		\indent 
		Deriving equation $(\ref{quadratura-35})$ with respect to $B$ and taking $x=0$ in the final result, we obtain by $(\ref{edo})$ at the point $x=0$ that 
		$\frac{\partial \phi}{\partial B}(0)=-\frac{1}{ 2\phi''(0)}$. In addition, since $\phi'$ is odd one has that $\frac{\partial \phi'}{\partial B}$ is also odd and thus, $\frac{\partial \phi'}{\partial B}(0)=0$. The existence and uniqueness theorem for classical ODE applied to the problem $(\ref{y})$ enables us to deduce that $\bar{y}=2\frac{\partial \phi}{\partial B}$. Therefore, we can combine $(\ref{theta1})$ with $(\ref{eq12343})$ to obtain that $\frac{dT}{dB}=-\frac{\theta}{2}$. 
	\end{proof}
	
	\indent By Proposition $\ref{prop2}$, Theorem $\ref{teo-uniq}$, Lemma $\ref{specprop}$, and Lemma $\ref{teo13}$, we obtain the following result about the non-positive spectrum of $\mathcal{L}_1$ defined in $(\ref{op1})$:
	\begin{lema}\label{lema167}
		Let $L>0$ be fixed and consider $\omega>\frac{1}{8}\left(\frac{\sqrt{L^2+16\pi}}{L}+\frac{8\pi}{L^2}-1\right)$. Operator
		$\mathcal{L}_{1}$
		in $(\ref{op1})$ defined in $L_{per}^2([0,L])$ with domain in
		$H_{per}^2([0,L])$ has a unique negative eigenvalue
		which is simple and zero is a simple eigenvalue with eigenfunction
		$\phi'$.  Moreover, the remainder of the spectrum of $\mathcal{L}_{1}$
		is constituted by a discrete set
		of eigenvalues bounded away from zero.
	\end{lema}
	\begin{flushright}
		$\square$
	\end{flushright}
	
	\indent The results established in Lemma $\ref{lema14}$ and Lemma $\ref{lema167}$ allow to conclude the following theorem concerning the behaviour of the non-positive spectrum of $\mathcal{L}$ defined in $(\ref{L-35i})$:
	\begin{teo}\label{teoL}
		Let $L>0$ be fixed and consider $\omega>\frac{1}{8}\left(\frac{\sqrt{L^2+16\pi}}{L}+\frac{8\pi}{L^2}-1\right)$. Operator
		$\mathcal{L}$
		in $(\ref{L-35i})$ defined in $\mathbb{L}_{per}^2 $ with domain in
		$\mathbb{H}_{per}^2$ has a unique negative eigenvalue
		which is simple and zero is a double eigenvalue with corresponding eigenfunctions
		$(\phi',0)$ and $(0,\phi)$.  Moreover, the remainder of the spectrum of $\mathcal{L}$
		is constituted by a discrete set
		of eigenvalues bounded away from zero.
	\end{teo}
	\begin{flushright}
		$\square$
	\end{flushright}
	
	\section{Orbital Stability - Proof of Theorem \ref{teo-nsl}.}\label{sec5}	
	
	We establish a result of orbital stability in this section. It is well known that (\ref{eq2}) has two basic symmetries, namely, translation and rotation. Indeed, if $u = u(x,t)$ is a solution of (\ref{eq2}), so are $e^{-i\theta}u$ and $u(x-r,t)$ for any  $\theta, r \in \mathbb{R}$. Considering $U=P+iQ\equiv(P,Q)$, we obtain that \eqref{eq2} is invariant under the transformations
	\begin{equation}\label{T1}
		S_1(\theta)U := \left(
		\begin{array}{cc}
			\cos{\theta} & \sin{\theta} \\
			- \sin{\theta} & \cos{\theta}
		\end{array}
		\right) \left(
		\begin{array}{c}
			P \\ Q
		\end{array}
		\right)
	\end{equation} 
	and
	\begin{equation}\label{T2}
		S_2(r)U := \left(
		\begin{array}{c}
			P(\cdot - r, \cdot) \\
			Q(\cdot - r, \cdot)
		\end{array}
		\right).
	\end{equation}
	The actions $S_1$ and $S_2$ define unitary groups in $\mathbb{H}^1_{per}$ with infinitesimal generators given by
	\begin{equation*}
		S_1'(0)U := \left(
		\begin{array}{cc}
			0 & 1 \\
			-1 & 0
		\end{array}
		\right) \left(
		\begin{array}{c}
			P \\
			Q
		\end{array}
		\right) = J \left(
		\begin{array}{c}
			P \\
			Q
		\end{array}
		\right)
	\end{equation*}
	and
	\begin{equation*}
		S_2'(0)U := \partial_x \left(
		\begin{array}{c}
			P \\
			Q
		\end{array}
		\right).
	\end{equation*}

	A standing wave solution  as in (\ref{PW2}) is then of the form
	\begin{equation*}
		U(x,t) = \left(
		\begin{array}{c}
			\phi(x) \cos(\omega t) \\
			\phi(x) \sin(\omega t)
		\end{array}
		\right).
	\end{equation*}  
	Taking into account that the NLS equation is invariant by the actions of  $S_1$ and $S_2$, we define the orbit generated by 
	$
	\Phi = (\varphi,0)
	$
	as
	\begin{equation*}
		\mathcal{O}_\Phi = \big\{S_1(\theta)S_2(r)\Phi; \theta, r \in \R \big\} = \left\{ \left(
		\begin{array}{cc}
			\cos{\theta} & \sin{\theta} \\
			-\sin{\theta} & \cos{\theta}
		\end{array}
		\right) \left(
		\begin{array}{c}
			\phi(\cdot - r) \\
			0
		\end{array}
		\right) \; ; \; \theta, r \in \R \right\}.
	\end{equation*}
	
	We introduce the pseudometric $d$ in $\mathbb{H}^1_{per}$ by
	\begin{equation*}
		d(f,g):= \inf \{ \|f - S_1(\theta)S_2(r)g\|_{\mathbb{H}^1}\; ; \; \, \theta, r \in \R\}.
	\end{equation*}
	In particular, the distance between $f$ and $g$ is the distance between $f$ and the orbit generated by $g$ under the action of rotation and translation. In particular, 
	$
	d(f,\Phi) = d(f,\mathcal{O}_\Phi).
	$

	We present now our notion of orbital stability.

	\begin{defi}\label{defstab}
		Let $L>0$ be fixed. Consider $\Theta(x,t) = (\phi(x)\cos(\omega t), \phi(x) \sin(\omega t))$ be a standing wave for (\ref{ham}), where $\omega>\frac{1}{8}\left(\frac{\sqrt{L^2+16\pi}}{L}+\frac{8\pi}{L^2}-1\right)$. We say that $\Theta$ is orbitally stable in $\mathbb{H}^1_{per}$ provided that, given $\varepsilon > 0$, there exists $\delta > 0$ with the following property: if $U_0 \in \mathbb{H}^1_{per}$ satisfies $\|U_0 - \Phi\|_{\mathbb{H}^1} < \delta$, then the local solution $U(t)$ defined in a local interval $[0,t_0)$ can be continued to a solution in $0\leq t<\infty$ and satisfies
		\begin{equation*}
			d(U(t), \mathcal{O}_\Phi) < \varepsilon, \ \ \ \ \ \text{for all} \; t \geq 0.
		\end{equation*}
		Otherwise, we say that $\Theta$ is orbitally unstable in $\mathbb{H}^1_{per}$.
	\end{defi}

	Important to notice that the definition of stability in $\ref{defstab}$ prescribes a local well-posedness result in the energy space $\mathbb{H}_{per}^1$. Thus, we have
	\begin{prop}
		Consider $U_0\in \mathbb{H}_{per}^s$. If $s>\frac{1}{2}$, then there exist a $t_0>0$ and a unique $U\in C([0,t_0);\mathbb{H}_{per}^s)$ such that $U$ solves the initial value problem
		\be
		\left\{\begin{array}{ll}
			iU_t+U_{xx}+|U|^2U+|U|^4U=0,\\
			U(x,0)=U_0(x).
		\end{array}  \right.
		\label{CPCQ}	\ee
		In addition,  the data mapping solution
		\begin{center}
			$U_0\in \mathbb{H}_{per}^s\longmapsto U \in C([0,t_0);\mathbb{H}_{per}^s)$
		\end{center}
		is continuous.
		\label{teolwp}\end{prop}
	\begin{proof}
		See \cite[Chapter 5]{Iorio}.
	\end{proof}
	
	\indent We have briefly mentioned in Remark $\ref{obs12}$ that the solution obtained in Proposition $\ref{teolwp}$ can not be extended to the whole semi-interval $[0,+\infty)$ because of the lack of \textit{a priori estimates} in the energy space $\mathbb{H}_{per}^s$, occasioned by the presence of the critical quintic power in $(\ref{CPCQ})$. In addition, it has been reported in the same remark that the pioneer work in \cite{Weinstein} (see also \cite{NP2015}) could be applied to obtain the orbital stability concerning the two symmetries $S_1$ and $S_2$ in $(\ref{T1})$ and $(\ref{T2})$, respectively, and by considering only standing wave solutions of the form $(\ref{PW2})$. To obtain the orbital stability using this method, the existence of global solutions in the energy space is a crucial point in the proof. Because of this, we shall use \cite{GSS} by considering only the rotation symmetry (that is, $r=0$ in the orbit $\mathcal{O}_{\Phi}$) and we prove the orbital stability in the restricted energy space $\mathbb{H}_{per,e}^1$ constituted by even periodic functions. With this restriction, Proposition $\ref{teolwp}$ can be considered in the new space $\mathbb{H}_{per,e}^1$ without further problems and Theorem $\ref{teoL}$ reads as follows:
	
	\begin{prop}\label{teoLe}
		Let $L>0$ be fixed and consider $\omega>\frac{1}{8}\left(\frac{\sqrt{L^2+16\pi}}{L}+\frac{8\pi}{L^2}-1\right)$. Operator
		$\mathcal{L}$
		in $(\ref{L-35i})$ defined in $\mathbb{L}_{per,e}^2$ with domain in
		$\mathbb{H}_{per,e}^2$ has a unique negative eigenvalue
		which is simple and zero is a simple eigenvalue with corresponding eigenfunction
		and $(0,\phi)$.  Moreover, the remainder of the spectrum of $\mathcal{L}$
		is constituted by a discrete set
		of eigenvalues bounded away from zero.
	\end{prop}
	\begin{flushright}
		$\square$
	\end{flushright}

	\indent  We proceed to the proof of Theorem $\ref{teo-nsl}$. Before that, we need a basic result:
	\begin{prop}\label{coro3}
		Let $L>0$ be fixed. We have that $\frac{dB}{d\omega}<0$ for all $\omega\in I_1$.
	\end{prop}
	\begin{proof}
		From $(\ref{B})$ and Theorem $\ref{teo1}$, we have $B=-\Lambda(\omega)\omega+\frac{\Lambda(\omega)^3}{3}+\frac{\Lambda(\omega)^2}{2}.$ In addition, $\Lambda$ is a positive smooth function defined in $I_1$ and
		
		\begin{equation}\label{lambdaprime}
			\frac{dB}{d\omega}=-\Lambda(\omega)-(\omega-\Lambda(\omega)^2-\Lambda(\omega))\Lambda'(\omega).
		\end{equation}
		By Corollary $\ref{coro1}$, we determined $\Lambda'(\omega)>0$ for all $\omega\in I_1$ and to prove that $\frac{dB}{d\omega}<0$, it is enough to establish $\omega-\Lambda(\omega)^2-\Lambda(\omega)>0$ for all $\omega\in I_1$. Indeed, for each $\omega\in I_1$, we obtain by Theorem $\ref{teo1}$ that $\Lambda(\omega)\in I_2\subset \left(\frac{ \sqrt{1+4\omega}-1}{2},\frac{\sqrt{48\omega+9}-3}{4}\right)$. The roots of the equation $\omega-\Lambda^2-\Lambda=0$ in terms of $\omega$ are $-1-\sqrt{4\omega+1}<0$ and $-1+\sqrt{4\omega+1}>0$. Since $0<\Lambda(\omega)<\frac{\sqrt{48\omega+9}-3}{4}<-1+\sqrt{4\omega+1}$, we see that $\omega-\Lambda^2(\omega)-\Lambda(\omega)>0$ for all $\omega\in I_1$, as desired.
	\end{proof}
	
	\indent Next, let $L>0$ be fixed. By Theorem $\ref{teo1}$, we obtain  $$\omega\in I_1\mapsto \Phi=(\phi,0)\in \mathbb{H}_{per,e}^n([0,L]),$$ is smooth for all $n\in\mathbb{N}$. Using Proposition $\ref{teoLe}$, we see that the number of negative eigenvalues is one and zero is a simple eigenvalue with associated eigenfunction $(0,\phi)$. According to the properties (i)-(iii) on page 3, we can decide if the periodic wave $\Phi$ is stable or not by calculating $d''(\omega)$ in $(\ref{d2})$. In fact, deriving equation $(\ref{edo})$ with respect to $\omega$, we obtain
	
	\begin{equation}\label{edoomega}
		-\eta''+\omega\eta-3\phi^2\eta-5\phi^4\eta=-\phi,
	\end{equation}
	where $\eta=\frac{d}{d\omega}\phi$ is clearly even. Multiplying equation $(\ref{edoomega})$ by $\phi$ and integrating over $[0,L]$, we deduce by $(\ref{edo})$
	\begin{equation}\label{edoomega1}
		\frac{1}{2}\frac{d}{d\omega}\int_0^L\phi^4dx+\frac{2}{3}\frac{d}{d\omega}\int_0^L\phi^6dx=\int_0^L\phi^2dx.
	\end{equation}
	In equation $(\ref{edo})$, we multiply it by $\phi$ and integrate the result over $[0,L]$ to get
	\begin{equation}\label{edoomega12}
		\frac{1}{2}\int_0^L\phi'^2dx+\frac{\omega}{2}\int_0^L\phi^2dx-\frac{1}{2}\int_0^L\phi^4dx-\frac{1}{2}\int_0^L\phi^6dx=0.
	\end{equation}	
	\indent On the other hand, integrating the quadrature form in $(\ref{quadratura-35})$ over $[0,L]$ and using $(\ref{edoomega12})$, we have
	\begin{equation}\label{edoomega123}
		2\omega\int_0^L\phi^2dx-\frac{3}{2}\int_0^L\phi^4dx-\frac{4}{3}\int_0^L\phi^6dx+BL=0.
	\end{equation}
	Deriving equation $(\ref{edoomega123})$ in terms of $\omega$, we conclude by $(\ref{edoomega1})$ that
	\begin{equation}\label{edoomega1234}
		2\omega\frac{d}{d\omega}\int_0^L\phi^2dx=\frac{1}{2}\frac{d}{d\omega}\int_0^L\phi^4dx-\frac{dB}{d\omega}L.
	\end{equation}
	To obtain a convenient expression for $\frac{d}{d\omega}\int_0^L\phi^4dx$, we need to multiply equation $(\ref{edo})$ by $\frac{1}{\phi}$ to get, after integration by parts
	\begin{equation}\label{1edoomega12}
		-\int_0^L\frac{\phi'^2}{\phi^2}dx+\omega L-\int_0^L\phi^2dx-\int_0^L\phi^4dx=0,
	\end{equation}		
	where we are using the fact $\phi>\sqrt{\alpha_2}>0$ to make sense the first term in $(\ref{1edoomega12})$.\\
	\indent Next, we can multiply the quadrature form in $(\ref{quadratura-35})$ by $\frac{1}{\phi^2}$ and integrate the result over $[0,L]$ to have by $(\ref{1edoomega12})$, the following equality
	\begin{equation}\label{2edoomega12}
		\frac{1}{2}\int_0^L\phi^2dx+\frac{2}{3}\int_0^L\phi^4dx+B\int_0^L\frac{1}{\phi^2}dx=0.
	\end{equation}		
	In equation $(\ref{2edoomega12})$, we derive it in terms of $\omega$ to obtain by $(\ref{edoomega1234})$, the final expression
	\begin{equation}\label{3edoomega1}
		\left(2\omega+\frac{3}{8}\right)\frac{d}{d\omega}\int_0^L\phi^2dx=-\frac{3}{4}\frac{dB}{d\omega}\int_0^L\frac{1}{\phi^2}dx-\frac{dB}{d\omega}L-\frac{3B}{4}\frac{d}{d\omega}\int_0^L\frac{1}{\phi^2}dx.
	\end{equation}	
	\indent Using Proposition $\ref{coro3}$, we obtain that the first two terms in the RHS of $(\ref{3edoomega1})$ are positive.  In fact, let us determine a convenient expression for $\int_0^L\frac{1}{\phi^2}dx$ using the explicit solution in $(\ref{sol2})$. Using the standard equalities of elliptic functions given by $k^2sn^2=1-dn^2$ and $k'nd=\sqrt{1-k^2}nd=dn(\cdot+K(k))$, we obtain
	\begin{equation}\label{expdn}\begin{array}{llll}
			\displaystyle \int_0^L\frac{1}{\phi^2}dx&=&\displaystyle\int_0^L\frac{1+\beta^2sn^2\left(\frac{2K(k)x}{L},k\right)}{\alpha_3dn^2\left(\frac{2K(k)x}{L},k\right)}\\\\
			&=&\displaystyle\frac{L}{K(k)\alpha_3}\int_0^K\left[\left(1+\frac{\beta^2}{k^2}\right)\frac{1}{dn^2(x,k)}-\frac{\beta^2}{k^2}\right]dx\\\\
			&=&\displaystyle\frac{L}{K(k)\alpha_3}\int_0^K\left[\left(1+\frac{\beta^2}{k^2}\right)nd^2(x,k)
			-\frac{\beta^2}{k^2}\right]dx\\\\
			&=&\displaystyle\frac{L}{K(k)\alpha_3k'^2}\left(1+\frac{\beta^2}{k^2}\right)\int_0^Kdn^2(x,k)dx
			-\frac{\beta^2L}{k^2\alpha_3}\\\\
			&=&\displaystyle\frac{L}{K(k)}\left(\frac{\alpha_1-\alpha_3}{k'^2\alpha_1\alpha_3}\right)E(k)+\frac{L}{\alpha_1},
	\end{array}\end{equation}
	where we are using the fact that $\beta^2=-\frac{\alpha_3}{\alpha_1}k^2$. Since the expression for $k^2$ in $(\ref{kg})$ establishes the basic equality  $\frac{\alpha_1-\alpha_3}{k'^2\alpha_1\alpha_3}=\frac{\alpha_1-\alpha_2}{\alpha_1\alpha_2}>0$, we obtain by $(\ref{expdn})$
	\begin{equation}\label{expdn1}
		\int_0^L\frac{1}{\phi^2}dx=\frac{L}{\alpha_1}-\frac{E(k)}{K(k)}\frac{L}{\alpha_1}+\frac{E(k)}{K(k)}\frac{L}{\alpha_2}.
	\end{equation}
	\indent We need to analyse the monotonicity of both terms  $\alpha_1^{-1}\left[1-\frac{E(k)}{K(k)}\right]$ and $\alpha_2^{-1}\frac{E(k)}{K(k)}$ with respect to $\omega$. First, we see that \be\label{expdn2}\begin{array}{lllll}\displaystyle\frac{d}{d\omega}\left(\alpha_1^{-1}\left[1-\frac{E(k)}{K(k)}\right]\right)&=&\displaystyle-\frac{\alpha_1'}{\alpha_1^2}+
		\frac{E(k)}{K(k)}\frac{\alpha_1'}{\alpha_1^2}\\\\
		&+&\displaystyle\alpha_1^{-1}\frac{d}{dk}\left[\frac{K(k)-E(k)}{K(k)}\right]\frac{dk}{d\omega},\end{array}\ee
	where $\alpha_1'$ represents the derivative of $\alpha_1$ with respect to $\omega$. By Remark $\ref{rem1}$, we see that $\alpha_1'>0$ and the first term in the RHS of $(\ref{expdn2})$ is negative while the second one is positive. To handle with the third term in the RHS, we notice that $\frac{d}{dk}\left[\frac{K(k)-E(k)}{K(k)}\right]>0$ for all $k\in(0,1)$ and $\frac{dk}{d\omega}<0$ by Corollary $\ref{coro1}$. The fact that $\alpha_1<0$ enables us to deduce $\alpha_1^{-1}\frac{d}{dk}\left[\frac{K(k)-E(k)}{K(k)}\right]\frac{dk}{d\omega}>0$. Next, we analyse the monotonicity of $\alpha_2^{-1}\frac{E(k)}{K(k)}$. In fact, since $k\in (0,1)\mapsto \frac{E(k)}{K(k)}$ is strictly decreasing, $\frac{d\alpha_2^{-1}}{d\omega}>0$, by Remark $\ref{rem1}$ and $\frac{dk}{d\omega}<0$, by Corollary $\ref{coro1}$, we obtain $\frac{d}{d\omega}\left[\alpha_2^{-1}\frac{E(k)}{K(k)}\right]>0$.\\
	\indent We see that the first term in the RHS of $(\ref{expdn2})$ is negative. Since $B$ is also negative, the product of them will be positive and by $(\ref{3edoomega1})$, the quantity $\frac{d}{d\omega}\int_0^{L}\phi^2dx$ can be eventually negative. Moreover, the first term $\frac{L}{\alpha_1}$ in the RHS of $(\ref{expdn1})$ is the only negative term and the remainder of the quantities are positive. Thus, by $(\ref{3edoomega1})$ we only need to handle with terms containing  factors of $\alpha_1$ which can become negative the value of $\frac{d}{d\omega}\int_0^L\phi^2dx$. First, we handle with $-\frac{\alpha_1'}{\alpha_1^2}$ in $(\ref{expdn2})$ and $B$ in $(\ref{Balpha1})$ (we omit the prefactor $L$ to simplify the notation). In fact, we have
	\be\label{R1}
	R_1=-\frac{3B}{4}\left(-\frac{\alpha_1'}{\alpha_1^2}\right)=-\frac{3}{4}\frac{\alpha_1'\omega}{\alpha_1}+\frac{\alpha_1\alpha_1'}{4}+\frac{3}{8}\alpha_1'.\ee
	Next, we deal with the second term in the RHS of $(\ref{3edoomega1})$ to get, by $(\ref{Balpha1})$,
	\be\label{R2}
	R_2=-\frac{dB}{d\omega}=-(-\alpha_1'\omega-\alpha_1+\alpha_1^2\alpha_1'+\alpha_1\alpha_1').
	\ee
	\indent Finally, we handle with $\frac{L}{\alpha_1}$ in the first term in the RHS of $(\ref{expdn1})$ and $\frac{dB}{d\omega}$. Omitting the prefactor $L$ by convenience, we obtain
	\be\label{R3}
	R_3=-\frac{3}{4}\frac{dB}{d\omega}\frac{1}{\alpha_1}=\frac{3}{4}\frac{\alpha_1'\omega}{\alpha_1}+\frac{3}{4}-\frac{3}{4}\alpha_1\alpha_1'-\frac{3}{4}\alpha_1'.
	\ee
	Thus, by $(\ref{R1})$, $(\ref{R2})$ and $(\ref{R3})$ we have
	\be\label{sumR}
	R=\sum_{i=1}^3R_i=-\frac{dB}{d\omega}+\frac{3}{4}-\alpha_1'\left(\frac{1}{2}\alpha_1+\frac{3}{8}\right).
	\ee
	For a fixed $L>0$, the value of  $\omega$ is positive. Since $\frac{\sqrt{48\omega+9}+3}{4}>\frac{3}{2}$, it follows that  $\alpha_1<-\frac{\sqrt{48\omega+9}+3}{4}<-\frac{3}{2}<-\frac{3}{4}$ and $\frac{1}{2}\alpha_1+\frac{3}{8}<0$. Thus $R$ in $(\ref{sumR})$ is positive and since the remainder of terms in $(\ref{expdn2})$ are positive, we obtain by $(\ref{3edoomega1})$ that $\frac{d}{d\omega}\int_0^L \phi^2dx>0$. Thus, $d''(\omega)>0$ and this fact finishes the proof of Theorem $\ref{teo-nsl}$.

	% For one-column wide figures use
	
	% For tables use

	%\begin{acknowledgements}
	%If you'd like to thank anyone, place your comments here
	%and remove the percent signs.
	%\end{acknowledgements}
	
	% BibTeX users please use one of
	%\bibliographystyle{spbasic}      % basic style, author-year citations
	%\bibliographystyle{spmpsci}      % mathematics and physical sciences
	%\bibliographystyle{spphys}       % APS-like style for physics
	%\bibliography{}   % name your BibTeX data base
	
	% Non-BibTeX users please use
	
	\section*{Author's contributions} All authors contributed equally to this work.

	\section*{Acknowledgments} F. Natali is partially supported by CNPq (grant 304240/2018-4), Funda\c{c}\~ao Arauc\'aria (grant 002/2017) and CAPES MathAmSud (grant 88881.520205/2020-01).

	\section*{Data Availability} Data sharing is not applicable to this article as no new data were created or analyzed in this study.

\end{document}